\documentclass[sn-basic]{sn-jnl}% Basic Springer Nature Reference Style/Chemistry Reference Style

\jyear{2021}%

\theoremstyle{thmstyleone}%
\newtheorem{theorem}{Theorem}[section]% meant for sectionwise numbers
\newtheorem{proposition}[theorem]{Proposition}% 
\newtheorem{lemma}[theorem]{Lemma}

\theoremstyle{thmstyletwo}%
\newtheorem{example}[theorem]{Example}%
\newtheorem{remark}[theorem]{Remark}%

\theoremstyle{thmstylethree}%
\newtheorem{definition}[theorem]{Definition}%

\newcommand{\IN}{\mathbb{N}}           % natural numbers
\newcommand{\IZ}{\mathbb{Z}}           % Integers 
\newcommand{\IR}{\mathbb{R}}           % real numbers

\newcommand{\mcS}{\mathcal{S}}          % state space
\newcommand{\mcA}{\mathcal{A}}          % action space
\newcommand{\mcR}{\mathcal{R}}          % set of rewards
\newcommand{\mcM}{\mathcal{M}}			% MDP
\newcommand{\mcMDP}{\mcM \doteq (\mcS, \mcA, \Psi, \mcR, p)}

\newcommand{\ol}[1]{ {\overline{#1}} }  % overline

\newcommand{\olS}{\ol{\mcS}}
\newcommand{\olA}{\ol{\mcA}}
\newcommand{\olPsi}{\ol{\Psi}}
\newcommand{\olR}{\ol{\mcR}}
\newcommand{\olp}{\ol{p}}

\newcommand{\quotsimM}{\mcM/\!\!\sim}
\newcommand{\quotsimMprime}{\mcM'/\!\!\sim}
\newcommand{\quotsimS}{\mcS/\!\!\sim}
\newcommand{\quotsimPsi}{\Psi/\!\!\sim}
\newcommand{\quotsimMprimeprime}{\mcM''/\!\!\sim}

\newcommand{\abs}[1]{\left\lvert #1 \right\rvert}       % absolut value
\newcommand{\myset}[1]{\left\{\, #1 \,\right\}}     % for sets
\newcommand{\mysetlong}[2]{\left\{\, #1 \,\middle\vert  \,  #2 \,\right\}}     % for sets with conditions. 
\newcommand{\longto}{\longrightarrow}

\DeclareMathOperator*{\argmax}{arg\,max}

\newcommand{\expected}[2][]{\mathbb{E}_{#1} \left[ #2 \right] } % expected value with optinal subtext
\newcommand{\expectedlong}[3][]{\mathbb{E}_{#1} \left[ #2 \middle\vert  #3 \right] } % expected value with
\newcommand{\reward}{\operatorname{Rew}}

\newcommand{\optdet}{ {\pi_{*}^{\det}} }

\begin{document}

%\title[Hidden symmetries in Markov decision processes]{Hidden symmetries and model minimization in Markov decision processes: explained and applied to the multi-period newsvendor problem}
%\title[Hidden Symmetries in Markov Decision Processes]{Hidden Symmetries in Markov Decision Processes: Explained and Applied to the Multi-period Newsvendor Problem}
\title[Hidden Symmetries in Markov Decision Processes]{Hidden Symmetries and Model Reduction in Markov Decision Processes: Explained and Applied to the Multi-period Newsvendor Problem}
%\title[Hidden Symmetries in Markov Decision Processes]{Reduction of Markov Decision Processes using Hidden Symmetries: Explained and Applied to the Multi-period Newsvendor Problem}


\author*[1]{\fnm{Tobias} \sur{Joosten}}\email{tobias.joosten@itwm.fraunhofer.de}

\author[1]{\fnm{Karl-Heinz} \sur{K\"ufer}}\email{karl-heinz.kuefer@itwm.fraunhofer.de}

\affil[1]{\orgdiv{Department of Optimization}, \orgname{Fraunhofer Institute for Industrial Mathematics ITWM}, \orgaddress{\street{Fraunhofer-Platz 1}, \city{Kaiserslautern}, \postcode{67663}, \state{Rhineland-Palatinate}, \country{Germany}}}


%\abstract{ 
%	Symmetry breaking is a common approach for model reduction of Markov decision processes (MDPs).
%	%	This approach only uses directly accessible/visible symmetries of MDPs.
%	%	For some MDPs, it is possible to transform them equivalently such that symmetries become visible -- we call this type of symmetries \emph{hidden symmetries}.
%	Currently, this approach only considers directly accessible symmetries of MDPs.
%	However, there are MDPs that have a symmetry that cannot be captured by this type of symmetries.
%	
%	For these MDPs, hidden symmetries allow substantially better model reduction compared to visible symmetries.	
%	The main idea is to reveal a hidden symmetry by altering the reward structure and then exploit the revealed symmetry by forming a quotient MDP.
%	The quotient MDP is the reduced MDP, since it is sufficient to solve the quotient MDP instead of the original one.
%	In this paper, we introduce hidden symmetries and the associated concept of model reduction.
%	We demonstrate this concept on the multi-period newsvendor problem, which is the newsvendor problem considered over an infinite number of days.
%	In this way, we show that hidden symmetries can reduce problems that visible symmetries cannot, and present a basic idea of revealing hidden symmetries in multi-period problems.
%	The presented approach can be extended to more sophisticated problems.
%}

\abstract{ 
	Symmetry breaking is a common approach for model reduction of Markov decision processes (MDPs).
	This approach only uses directly accessible symmetries such as geometric symmetries.
	For some MDPs, it is possible to transform them equivalently such that symmetries become accessible -- we call this type of symmetries \emph{hidden symmetries}.
	For these MDPs, hidden symmetries allow substantially better model reduction compared to directly accessible symmetries.	
	The main idea is to reveal a hidden symmetry by altering the reward structure and then exploit the revealed symmetry by forming a quotient MDP.
	The quotient MDP is the reduced MDP, since it is sufficient to solve the quotient MDP instead of the original one.
	In this paper, we introduce hidden symmetries and the associated concept of model reduction.
	We demonstrate this concept on the multi-period newsvendor problem, which is the newsvendor problem considered over an infinite number of days.
	In this way, we show that hidden symmetries can reduce problems that directly accessible symmetries cannot, and present a basic idea of revealing hidden symmetries in multi-period problems.
	The presented approach can be extended to more sophisticated problems.
}

\keywords{Markov decision process, model reduction, hidden symmetry, newsvendor problem, multi-period problem}



\maketitle

\section{Introduction}

A \emph{Markov decision process} (MDP) is a stochastic framework that models the interaction of a decision maker with an environment.
The environment is in a certain state and the decision maker has to take an action. The outcome of the taken action is a new state and a reward, which are partially random and only depend on the current state.
The objective is to compute an policy/behavior for the decision maker that maximizes the expected total discounted reward; such a policy/behavior is called \emph{optimal}.
This computation can be done using algorithms from \emph{dynamic programming} or \emph{reinforcement learning} depending on what information about the MDP is given, as stated by \cite{Sutton2017}.
Many real-world problems are too complex to be solved in a reasonable amount of time using these algorithms.
For example,
\cite{Boutilier1995} discussed this for planning problems where the number of states grows exponentially with the number of variables relevant to the problem.
This is especially due to the fact that the common algorithms for solving them run in time polynomial in the size of the state space \citep{Puterman2014}.
% as \citet{Givan2003} stated. %% they refer to Puterman
Therefore, if you want to solve these problems, you have to reduce them.
An elementary and theoretical question that arises is: What options do we have to reduce MDPs?

The simplest way is to aggregate states in an MDP that can simulate each other. This general concept is called \emph{bisimulation}.
\cite{Dean1997} did this by using partitions of the state space that fulfill \emph{stochastic bisimulation homogeneity}, as their framework. This property is related to the \emph{substitution property} for finite automata \citep{Hartmanis1966}, the \emph{lumpability} for Markov chains \citep{Kemeny1976} and \emph{bisimulation equivalence} for transition systems \citep{Lee1992}.
\cite{Givan2003} did this by using \emph{stochastic bisimulation relations}, which is a general concept for transition systems \citep{Castro2010,Larsen1991}.
The weakness of bisimulation is that the naming of actions has a strong impact on which states are equivalent to each other.
Therefore, the naming of actions affects which states are aggregated together.
For example, the games tic-tac-toe and Go contain rotation and mirror symmetries. However, they cannot be reduced using stochastic bisimulation because of the name of the actions.
Here, the concept of \emph{symmetries} is needed.

Symmetries are a more general concept than bisimulation. 
They can be used to reduce MDPs similar to symmetry breaking in combinatorial programming \citep{Walsh2006, Walsh2012}.
%Some obvious examples for problems that can be reduced using symmetries are games like tic-tac-toe or Go. In these, one can use geometric symmetries like rotation or mirror symmetry to aggregate states of the board; to be more precise, state-action pairs are aggregated here. 
There are several approaches to model symmetries in MDPs: 
\cite{Zinkevich2001} defined particular equivalence relations to model symmetries. 
\cite{Ravindran2001} took a different approach and introduced homomorphisms between MDPs to define symmetries.
These two definitions of symmetries are equivalent.
%The symmetries of Zinkevich and Balch as well as the symmetries of Ravidran and Barto,
They include only the directly accessible/visible symmetries in an MDP 
and not the symmetries that are visible after an equivalent transformation of the MDP.
We refer to this type of symmetry as \emph{visible symmetry}.
As an example, \cite{Mahajan2017} investigated reflection symmetry in the Cart-Pole problem or fold symmetries in grid worlds. These are classical examples for visible symmetries.

Bisimulation and visible symmetries aggregate states or state-action pairs of an MDP that have the same reward structure.
These aggregations partition the state space or set of state-action pairs.
\cite{Dean1997} and \cite{Ravindran2001} stated that these partitions induce \emph{quotient MDPs}.
This is basically about the transitions between blocks of aggregated states, since the reward structure is the same in each block.
Quotient MDPs have the property that solutions to them can be transferred to solutions to the original MDP.
This is done using a \emph{pullback}, a mapping that maps an optimal policy in the quotient MDP to an optimal policy in the original MDP \citep{Ravindran2001, Givan2003}.
%\todo{they call it lift, but lift is not the right terminus}
The quotient MDP in combination with the pullback forms the \emph{reduction} of the original MDP.
In addition, there is a maximally reduced MDP, giving rise to the concept of model \emph{minimization}, which was discussed by \cite{Dean1997} with respect to bisimulation and \cite{Ravindran2001} with respect to visible symmetries.

Overall, bisimulation and visible symmetries are our current options for reducing MDPs.
%They are based on the fact that the reward structure for equivalent states or equivalent state-action pairs have to be equal.
However, there are problems that have no visible symmetries but where a symmetry is hidden.
Therefore, their MDPs cannot be reduced using the above techniques.
In this paper, we introduce a new type of symmetry that captures these hidden symmetries and allows us to reduce these problems:

\paragraph*{Hidden Symmetries}
The main idea is to reveal hidden symmetries by transforming the given MDP.
Two types of transformations are allowed: a) altering the reward structure while keeping optimal policies, and b) relabeling of states and actions.
Since these transformations do not change the main structure of the MDP,
we call MDPs \emph{equivalent} if they can be transformed into each other by these transformations.
Indeed, there is a bijective mapping between the sets of optimal policies of equivalent MDPs.

A \emph{hidden symmetry} of an MDP $\mcM$ is basically a visible symmetry of an equivalent MDP. Therefore, it is formally defined as a tuple $(\mcM', \sim)$ where $\mcM'$ is equivalent to $\mcM$ and $\sim$ is a visible symmetry of $\mcM'$.
This tuple naturally induces the associated quotient MDP $\quotsimMprime$ and a pullback of optimal policies from $\quotsimMprime$ to $\mcM'$.
This means, it is sufficient to solve $\quotsimMprime$ instead of $\mcM'$.
Furthermore, $\mcM$ and $\mcM'$ are equivalent.
%, and thus
Thus, there is bijective mapping between their optimal policies.
This yields a pullback from optimal policies in $\quotsimMprime$ to optimal policies in $\mcM$.
This means, it is sufficient to solve $\quotsimMprime$ instead of $\mcM$.
Altogether, the quotient MDP $\quotsimMprime$ in combination with this pullback forms the reduction of the original MDP $\mcM$ using the hidden symmetry $(\mcM',\sim)$.
This approach describes the concept of \emph{model reduction using hidden symmetries}. 
Model reduction can be extended to model \emph{minimization}, but we do not discuss this topic in this paper.

Hidden symmetries are an extension of visible symmetries because for trivial transformations hidden symmetries are basically visible symmetries.
Therefore, hidden symmetries extend the model reduction and model minimization using visible symmetries of \cite{Zinkevich2001} and \cite{Ravindran2001}.
%\\

The most important part of hidden symmetries is the transformation that alters the reward structure while keeping optimal policies. This transformation causes hidden symmetries to extend visible symmetries.
%because for trivial transformations hidden symmetries are basically visible symmetries.
This type of transformation is called \emph{policy invariant reward shaping} and was introduced by \cite{Ng1999}. 
The more general concept of altering the reward structure without regard to optimal policies is called \emph{reward shaping}.
However, reward shaping is a transformation that is not suitable for hidden symmetries. This is because reward shaping in general results in a non-equivalent MDP that provides incorrect solutions to the original problem, as showed by \cite{Randlov1998}.
In the literature, the general focus and motivation of (policy invariant) reward shaping is to speed up the learning process of reinforcement learning algorithms
\citep{Ng1999, Behboudian2021, Laud2003, Laud2004}.
In contrast, our motivation is to reveal a hidden symmetry to reduce the given MDP.
%using policy invariant reward shaping, which can be used to reduce the given MDP.
%\\

Hidden symmetries exist in various problems.
Since factored MDPs \citep{Boutilier2000} received some attention in the literature on model minimization in terms of bisimulation \citep{Dean1997} and visible symmetries \citep{Givan2003},
we are interested in a similar field where hidden symmetries exist.
It appears that control problems in discrete-time dynamical systems 
is such a field.
%should benefit from our concept of hidden symmetries.
These problems are basically multi-period problems that operate deterministically and are perturbed stochastically from the outside \citep{Puterman2014}. They arise frequently in various domains such as supply chains, logistics and other economic sectors.
A simple 
%and well-know 
problem of this problem class is the \emph{multi-period newsvendor problem}, which is an extension of the newsvendor problem that was first introduced by \citet{Edgeworth1888}.  
The newsvendor problem is a classical problem in the literature and is often used for analysis (see \cite{Khouja1999} and \cite{Qin2011} for review articles).

Our ultimate goal is to use the multi-period newsvendor problem to show that hidden symmetries exist, how they can be revealed, and that they allow us to reduce problems.
This problem is well-suited for this demonstration because it has no visible symmetries but a hidden symmetry.
%\\

This paper is organized as follows.
Section \ref{section - basics and notations} introduces basics about MDPs and the required tools needed for revealing and exploiting hidden symmetries.
This includes visible symmetries in MDPs and their quotient MDPs, reward shaping and relabeling.
Then, Section \ref{section - hidden symmetries} introduces hidden symmetries and the main theorem of this paper that shows how to reduce an MDP using a hidden symmetry.
The concept of reducing an MDP using a hidden symmetry is then demonstrated in Section \ref{section - reduction of the m-p newsvendor problem} on the multi-period newsvendor problem.
In addition, we briefly discuss the multi-period newsvendor problem with a 5-day cycle and show that our reduction method works here as well.
We finish with a conclusion in Section \ref{section - conclusion}.

\section{Preliminaries}\label{section - basics and notations}
This section introduces the \emph{Markov Decision Process} (MDP) in its very general form defined by \cite{Sutton2017} plus the tools that are needed to define hidden symmetries in MDPs.

\subsection{Markov Decision Process}
A \emph{Markov Decision Process} is a 5-tuple $\mcM \doteq(\mcS, \mcA, \Psi, \mcR, p)$ where 
$\mcS$ is the set of states,
$\mcA$ is the set of actions,
$\Psi \subseteq \mcS \times \mcA$ is the set of admissible state-action pairs such that for all $s\in\mcS$ there exists an $a\in\mcA$ with $(s,a)\in \Psi, $
$\mcR \subseteq \IR$ is a bounded subset and is the set of rewards, and
$p$ is the dynamic of the MDP given as a map $p\colon \mcS \times \mcR \times \Psi \to [0,1]$ so that $\sum_{s'\in\mcS,r \in\mcR} p(s',r,(s,a)) = 1$ holds for all $(s,a)\in\Psi$ .
%$p$ is a map with $p\colon \mcS \times \mcR \times \Psi \to [0,1]$ that holds the property
%\begin{align*}
%	\forall (s,a) \in\Psi :\sum_{s'\in\mcS,r \in\mcR} p(s',r,(s,a)) = 1,
%\end{align*} 
%which is called the dynamic of the MDP $\mcM$. 	
The value $p(s',r,(s,a))$ describes the probability of going 
to state $s'$ and receiving reward $r$ when starting in state $s$ and taking action $a$;
thus, we use the notation $p(s',r \vert s,a) \doteq p(s',r,(s,a))$.		
The state transition function is encoded in the dynamic $p$. It is given by the function
$\mcS \times \Psi  \to [0,1]$, $(s',(s,a)) \mapsto \sum_{r\in\mcR} p(s',r\vert s,a)$ that we also denote by the letter $p$.
%\begin{align*}
%	\mcS \times \Psi  \longto [0,1], \quad
%	(s',(s,a)) \longmapsto \sum_{r\in\mcR} p(s',r\vert s,a).		
%\end{align*}
%We denote this function also by the letter $p$. 
%Since the state transition function is a three-argument function and the dynamic is a four-argument function, one can distinguish between the dynamic and the state transition function by the arguments.
Because $p (s',(s,a))$ describes the probability of going from state $s$ to state $s'$ with action $a$, we have a conditional probability and use the notation 
$p(s'\vert s,a) \doteq p (s',(s,a))$.
Furthermore, we denote by $\mcA(s) \doteq \mysetlong{a \in\mcA }{ (s,a) \in\Psi}$ the admissible actions in the state $s\in\mcS$.
In this work, we assume that the set $\mcS$ is countable and $\mcA(s)$ is finite for all $s\in\mcS$.

A \emph{policy} $\pi$ is a mapping from $\Psi$ to the real interval $[0,1]$ such that for all $s\in\mcS$ the equation $	\sum_{a\in\mcA(s)} \pi((s,a)) = 1$ is true.
The value $\pi((s,a))$ describes the probability of taking action $a$ in state $s$; therefore, we use the notation $\pi(a\vert s) \doteq \pi((s,a))$.	
Furthermore, we denote the \emph{set of all polices in $\mcM$} by $\Pi(\mcM)$.
A policy $\pi$ is called \emph{deterministic} if it only takes values in $\myset{0,1}$. Then, there exists for all $s\in\mcS$ exactly one $a\in\mcA$ with $\pi(a\vert s)=1$.
Therefore, we can write a deterministic policy $\pi$ also as a function from $\mcS$ to $\mcA$ with $\pi(s)=a$ for all $(s,a)\in\Psi$ with $\pi(a\vert s) =1$.

The \emph{trajectory} that is induced in an MDP by a policy starting in state $S_0$ is denoted by $(S_0,A_0,R_1,S_1,A_1,R_2,  \dots)$ where $S_i$, $A_i$, $R_i$ are random variables of state, action and reward, respectively.

% optimality criterion:
As the optimality criterion we use the expected total discounted reward. 
Therefore, the \emph{value function} for a given policy $\pi$ in the MDP $\mcM$ is given by
$V_\pi^\mcM \colon \mcS \to \IR$ with
$V_\pi^\mcM(s) \doteq  \expectedlong[\pi,p]{ \sum_{k=0}^\infty \gamma^k R_{k+1} }{ S_0 = s}.$

The \emph{optimal value function} is defined as $V_*^\mcM (s) \doteq \sup_{\pi\in\Pi(\mcM)} V_\pi^\mcM(s)$, and an \emph{optimal policy} is defined as a policy $\pi$ with $V_\pi^\mcM(s)=V_*^\mcM(s)$ for all $s\in\mcS$.
We denote the \emph{set of all optimal policies in $\mcM$} by $\Pi_*(\mcM)$.

Lastly, we define the \emph{expected reward of taking action $a$ in state $s$} by 
$R^\mcM(s,a)$.

\subsection{Visible symmetries in MDPs and their Quotient MDPs}

This subsection defines visible symmetries in MDPs based on the definition of MDP symmetries by \cite{Zinkevich2001}.
The term visible symmetry is used to better distinguish between visible and hidden symmetries.
%Additionally, we define when a visible symmetry is called simple.
It is also shown how to create a quotient MDP out of an MDP and an associated visible symmetry, and how to pull policies from the quotient MDP back to the original MDP.
%It also takes a look at the case where $\sim$ is a simple visible symmetry; here several things become simpler.

\begin{definition}	
Let $\mcMDP$ be an MDP.
A \emph{visible symmetry} of $\mcM$ is a tuple $\sim \doteq (\sim_\mcS, \sim_\Psi)$ where
	$\sim_\mcS$ is an equivalence relation of $\mcS$, and	
	$\sim_\Psi$ is an equivalence relation of $\Psi$			
	such that
	\begin{itemize}
		\item for all $(s_1,a_1)\in\Psi$ and $s_2\in\mcS$ with $s_1\sim_\mcS s_2$ exists an action $a_2\in\mcA(s_2)$ with $(s_1,a_1) \sim_\Psi (s_2,a_2)$, and
		
		\item for all $(s_1,a_1), (s_2,a_2)\in\Psi$ with $(s_1,a_1) \sim_\Psi (s_2,a_2)$ it holds $s_1 \sim_\mcS s_2$	and
		 the statement
		\begin{align*}
			\forall X\in\mcS/\!\!\sim_\mcS, r\in\mcR : 
			p(X,r \vert s_1,a_1) = p(X,r\vert s_2,a_2)
		\end{align*}
		with $p(X,r\vert s_i,a_i) \doteq \sum_{\tilde{s}\in X} p(\tilde{s},r\vert s_i,a_i)$
		is true.
	\end{itemize}
	Additionally, we call a visible symmetry \emph{simple} if 
	$(s_1,a_1) \sim_\Psi (s_2,a_2)$ implies	$a_1 = a_2$.
%	for all $(s_1,a_1)$, $(s_2,a_2) \in\Psi$ with $(s_1,a_1) \sim_\Psi (s_2,a_2)$ the equation	$a_1 = a_2$	holds.

Since a visible symmetry of an MDP can identify several actions in a state with each other, we capture the number of actions identified with each other for a state-action pair $(s,a)\in\Psi$ by 
$N_\Psi (s,a) \doteq \abs{\mysetlong{a'\in\mcA(s) }{ (s,a) \sim_\Psi (s,a')} }$.
\end{definition}

\begin{remark}
	Simple visible symmetries are basically stochastic bisimulations. 
	For these, $N_\Psi(s,a)$ is always equal to 1.
\end{remark}

\begin{definition}
	A visible symmetry $\sim \doteq (\sim_\mcS, \sim_\Psi)$ of the MDP $\mcM$ induces the \emph{quotient MDP} $\quotsimM$, which is defined
	(by abuse of notation) as the MDP $\quotsimM \doteq(\ol{\mcS}, \mcA, \ol{\Psi}, \mcR, \ol{p} )$ 
	with $\ol{\mcS} \doteq \quotsimS_\mcS$, $\ol{\Psi} \doteq \Psi/\!\!\sim_\Psi$
	and the dynamic
	$	\ol{p}\colon \ol{\mcS} \times \mcR \times \ol{\Psi} \to [0,1]	$
	is given by
	\begin{align*}
		\ol{p}([s'],r\vert [(s,a)]) &\doteq p([s'],r \vert s,a)\doteq \sum_{\tilde{s}\in[s]} p(\tilde{s},r\vert s,a).				
	\end{align*}
	If we have a simple visible symmetry, the quotient MDP simplifies and we can write 
	$\ol{p}([s'],r\vert [s],a)$ instead of $\ol{p}([s'],r\vert [(s,a)])$ and $\pi(a\vert [s])$ instead of $\pi([(s,a)])$ where $\pi$ is a policy in $\quotsimM$.

	The abuse of notation refers to the fact that by definition $\olPsi \subseteq \olS \times \mcA$ must hold.
	This is fine because by choosing representatives $[(s,a)]$ for all elements in $\olPsi$ and using the mapping $[(s,a)] \mapsto  ([s],a)$, we get $\olPsi \subseteq \olS \times \mcA$. 
	However, such a determination of representatives is not necessary for the quotient MDP.
\end{definition}

Policies in the quotient MDP can be \emph{pulled back} to policies to the original MDP. Moreover, the pullback preserves the optimality property.
Therefore, it is sufficient to solve the quotient MDP instead of the original one. We refer to a theorem of \citet{Ravindran2001} for this:

\begin{theorem}\label{thm - quotient MDP pullback of optimal policies}
	Let $\mcMDP$ be an MDP and $\sim$ a visible symmetry of $\mcM$.
	Let $\pi$ be a policy in the quotient MDP $\quotsimM$.
	The pullback of the policy $\pi$ is the policy $\pi'\in\Pi(\mcM)$ with
	\begin{align*}
		\pi' \colon \Psi \longto [0,1], \quad
		(s,a) \longmapsto \frac{\pi([(s,a)]_\Psi)}{N_\Psi(s,a)}.
	\end{align*}
	If $\pi$ is optimal in $\quotsimM$, then the pullback $\pi'$ is optimal in $\mcM$.	
	Moreover, if $\sim$ is a simple visible symmetry, the pullback simplifies to $\pi'(s,a) = \pi(a \vert [s])$.
\end{theorem}
\begin{proof}
	Ravindran and Barto showed this only for finite MDPs using MDP homomorphisms \citep[Theorem 2]{Ravindran2001}. However, their result can be extended to our kind of MDP.
	Finally, by translating the quotient map 
	\[ 	\Psi \longto \quotsimPsi,\quad	(s,a) \longmapsto [(s,a)]  \]
	into an MDP homomorphism, we can apply their statement and the theorem follows.
\end{proof}

 \begin{remark}
 	Note that it is sufficient to define the pullback of $\pi$ such that the equation
 	$$ \sum_{a'\in \mysetlong{a'\in\mcA(s) }{ (s,a)\sim_\Psi (s,a')}}
 	\pi'(s,a')  = \pi([(s,a)]_\Psi)$$
 	is true for all $(s,a)\in\Psi$. We use the definition above to make the pullback unique.
 \end{remark}

\subsection{Policy Invariant Reward Shaping}

This subsection introduces reward shaping as well as policy invariant reward shaping.
These concepts are converted into concepts of transition structures of MDPs, as they are more practical.
\\

\emph{Reward shaping} is the technique to alter the dynamic of an MDP without changing its transition structure; thus, only the reward structure is changed.
From another point of view, this means: 
Two MDPs can create each other by reward shaping if and only if they have the same transition structure (this includes that the set of states, actions and admissible state-action pairs are the same).
Therefore, we define the following property for two MDPs.
\begin{definition}
	Let $\mcMDP$ and $\mcM' \doteq (\mcS,\mcA, \Psi, \mcR',p')$ be MDPs. 
	The MDPs $\mcM$ and $\mcM'$ \emph{have the same transition structure} if the statement
	\[ \forall (s,a)\in\Psi, s'\in\mcS : p(s' \vert s,a) = p'(s'\vert s,a) \]
	is true.
\end{definition}

\emph{Policy invariant reward shaping} is reward shaping while keeping optimal policies.
%such that the optimal policies are not changed.
From another point of view, this means: 
Two MDPs $\mcM$ and $\mcM'$ can create each other by policy invariant reward shaping if and only if they have the same transition structure and the same optimal policies $\Pi_*(\mcM) = \Pi_*(\mcM')$.

\subsection{Relabeling}

This subsection introduces relabeling and the associated transfer of optimal policies.

\begin{definition}
A \emph{relabeling} of an MDP $\mcMDP$ to another MDP $\mcM' \doteq (\mcS', \mcA', \Psi', \mcR, p')$ is a tuple $h\doteq(f,\mysetlong{g_s }{ s\in\mcS})$ such that
\begin{align*}
	f&\colon \mcS \longto \mcS' \\
	g_s &\colon \mcA(s) \longto \mcA'(f(s)), \quad s\in\mcS
\end{align*}
are bijective functions and the equation
\begin{align*}
	p(\tilde{s},r\vert s,a) = p'(f(\tilde{s}), r \vert  f(s), g_s(a))
\end{align*}
is true for all $(s,a)\in\Psi$, $\tilde{s}\in\mcS$ and $r\in\mcR$.
We then say that the MDP $\mcM'$ is a \emph{relabeled variant of} $\mcM$.
\end{definition}

%Since a relabelling is only a renaming of states and actions without changing any structure of the MDP, we obtain the following proposition about the relationship of optimal policies in $\mcM$ and $\mcM'$.
The following statement about the relationship of optimal policies in $\mcM$ and $\mcM'$ holds:

\begin{proposition}\label{prop - optimal policies relabelling}
	Let $\mcMDP$ and $\mcM' \doteq (\mcS', \mcA', \Psi', \mcR, p')$ be MDPs and $h\doteq(f,\mysetlong{g_s }{ s\in\mcS})$ be a relabeling of $\mcM$ to $\mcM'$.
%	Then, the mapping 
	The mapping
	\[\Pi_*(\mcM') \longto \Pi_*(\mcM), \quad \pi_* \longmapsto (\pi_* \circ h)  \]	
	with $(\pi_* \circ h) (s,a) \doteq \pi_* (g_s(a)\vert f(s))$ for all $(s,a)\in\Psi$ is bijective.
\end{proposition}
\begin{proof}
	The proposition follows from the fact that the relabeling $h$ is only a renaming of states and actions without changing any structure of the MDP.
\end{proof}

\section{Hidden Symmetries}\label{section - hidden symmetries}

This subsection introduces hidden symmetries of MDPs and shows that it is sufficient to solve the associated quotient MDP of a hidden symmetry instead of the original MDP.

%We begin by defining equivalence for MDPs.
We start with the definition of equivalence for MDPs.

\begin{definition}\label{definition - equivalente MDPs}
	We call two MDPs $\mcM$ and $\mcM'$ \emph{equivalent} if
	\begin{enumerate}
		\item $\mcM'$ is just a relabeled variant of $\mcM$, or
		\label{item - equivalent MDP 1}
		
		\item $\mcM$ and $\mcM'$ have the same transition structure and the same optimal policies (the reward structure can be different!), or
		\label{item - equivalent MDP 2}
		
		\item 
		it exists an MDP $\tilde{\mcM}$ such that $\tilde{\mcM}$ is a relabeled variant of $\mcM$, and $\tilde{\mcM}$ and $\mcM'$ have the same transition structure and the same optimal policies, or
		\label{item - equivalent MDP 3}
		
		\item		
		it exists an MDP $\tilde{\mcM}$ such that
		$\mcM$ and $\tilde{\mcM}$ have the same transition structure and the same optimal policies, and $\mcM'$ is a relabeled variant of $\tilde{\mcM}$.	
		\label{item - equivalent MDP 4}		
	\end{enumerate}
	(Statements \ref{item - equivalent MDP 3} and \ref{item - equivalent MDP 4} result from the combination of \ref{item - equivalent MDP 1} and \ref{item - equivalent MDP 2}.)
	
	In Statement \ref{item - equivalent MDP 1}, \ref{item - equivalent MDP 3} and \ref{item - equivalent MDP 4}, there is a relabeled variant. Thus, there is a relabeling of MDPs $h$.
	In Statement \ref{item - equivalent MDP 2}, we set $h$ trivial. 
	This relabeling $h$ is a relabeling of the set of state-action pairs of $\mcM$ to the set of state-action pairs of $\mcM'$.
	Therefore, we call $h$ \emph{relabeling of the sets of state-action pairs in respect to the equivalent definition of MDPs $\mcM$ and $\mcM'$}.
	In general, $h$ is no relabeling of $\mcM$ to $\mcM'$.
\end{definition}

Using this definition, we can introduce hidden symmetries.

\begin{definition}
	Let $\mcM$ be an MDP.
	A tuple $(\mcM', \sim)$ is called \emph{hidden symmetry of $\mcM$} if $\mcM$ and $\mcM'$ are equivalent and $\sim$ is a visible symmetry of $\mcM'$.
\end{definition}

Equivalent MDPs have the following relationship between their optimal policies.

\begin{lemma}\label{lemma - equivalent MDPs have relabbeling of Psi to Psi' such that optimal policies are preserved}
	Let $\mcMDP$ and $\mcM' \doteq (\mcS', \mcA',\Psi',\mcR',p')$ be equivalent MDPs
	%	Then, there exists a relabelling of $\Psi$ to $\Psi'$ $h \doteq (f,\mysetlong{g_s}{ s\in\mcS})$ such that the mapping
	%	\[ \Pi_*(\mcM') \longto \Pi_*(\mcM), \quad \pi_* \longmapsto (\pi_*\circ h ) \]
	%	is bijective.
	%	Recall that $(\pi_*\circ h )(s,a)\doteq \pi_*(g_s(a) \vert f(s))$ holds for all $(s,a)\in\Psi$.	
	%	
	and let $h \doteq (f, \mysetlong{g_s}{s\in\mcS})$ be a relabeling of $\Psi$ to $\Psi'$ in respect to the equivalent definition of the MDPs $\mcM$ and $\mcM'$.
	Then, the mapping
	\[ \Pi_*(\mcM') \longto \Pi_*(\mcM), \quad \pi_* \longmapsto (\pi_*\circ h ) \]
	is bijective.
	Recall that $(\pi_*\circ h )(s,a)\doteq \pi_*(g_s(a) \vert f(s))$ holds.		
\end{lemma}
\begin{proof}
%	Let $\mcMDP$ and $\mcM' \doteq (\mcS', \mcA',\Psi',\mcR',p')$ be equivalent MDPs. Furthermore, let $h$ be the relabelling of $\Psi$ to $\Psi'$ in respect to the equivalent definition of the MDPs $\mcM$ and $\mcM'$.
	We define the map
	\[ \Phi \colon \Pi_*(\mcM') \longto \Pi_*(\mcM), \quad \pi_* \longmapsto (\pi_*\circ h ). \] 
	We have to show that $\Phi$ is bijective.
	
	Since the MDPs $\mcM$ and $\mcM'$ are equivalent, one of the four statements of Definition \ref{definition - equivalente MDPs} holds.
	
	If Statement \ref{item - equivalent MDP 1} holds, then $\mcM'$ is a relabeled variant of $\mcM$ and $h$ is a relabeling of these MDPs. 
	By applying Proposition \ref{prop - optimal policies relabelling}, we obtain that
	$\Phi$ is bijective.

	If Statement \ref{item - equivalent MDP 2} holds, then $\mcM$ and $\mcM'$ have the same optimal policies, the equation $\Psi=\Psi'$ is true and $h$ is trivial. 
	Hence, the map $\Phi$ is bijective.
	
	Since the Statements \ref{item - equivalent MDP 3} and \ref{item - equivalent MDP 4} are combinations of the statements \ref{item - equivalent MDP 1} and \ref{item - equivalent MDP 2},
	the map $\Phi$ is bijective.		
\end{proof}

Now, we can state the main theorem of this paper.

\begin{theorem}\label{thm - hidden symmetry pullback of optimal policy}
	Let $\mcMDP$ and $\mcM' \doteq (\mcS', \mcA', \Psi', \mcR', p')$ be MDPs.
	Furthermore, let $(\mcM', \sim)$ be a hidden symmetry of $\mcM$.	
	Every optimal policy in the quotient MDP $\quotsimMprime$ can be pulled back to an optimal policy in the MDP $\mcM$.
	
	To be more precise: 
	Let $h=(f,\mysetlong{g_s }{ s\in\mcS})$ be a relabeling from $\Psi$ to $\Psi'$ in respect to the equivalent definition of the MDPs $\mcM$ and $\mcM'$.
	The associated pullback of a policy $\pi\in\Pi(\quotsimMprime)$ to a policy $\pi' \in\Pi(\mcM)$ is defined as
	\begin{align*}
		\pi'(s,a) 	
		\doteq 
		\frac{\pi([(f(s),g_s(a))]_\Psi)}{N_\Psi(f(s),g_s(a))}
	\end{align*}
	for all $(s,a)\in\Psi$.
	If $\pi$ is optimal in $\quotsimMprime$, then $\pi'$ is optimal in $\mcM$.
	
	Moreover, if $\sim$ is simple, the pullback simplifies to $\pi'(s,a) \doteq \pi(g_s(a)\vert [f(s)])$.
\end{theorem}
\begin{proof}
	By combining Theorem \ref{thm - quotient MDP pullback of optimal policies} and Lemma \ref{lemma - equivalent MDPs have relabbeling of Psi to Psi' such that optimal policies are preserved}, this theorem follows.
\end{proof}
	
Theorem \ref{thm - hidden symmetry pullback of optimal policy} shows that it is sufficient to solve the associated quotient MDP of an hidden symmetry instead of the original MDP.
It indirectly provides an injective mapping from $\Pi_*(\quotsimMprime)$ to $\Pi_*(\mcM)$, which completes the reduction using hidden symmetries.

\section{Reduction of the Multi-Period Newsvendor Problem using a Hidden Symmetry}
\label{section - reduction of the m-p newsvendor problem}

This section applies the concept of model reduction using hidden symmetries to the multi-period newsvendor problem. 
The multi-period newsvendor problem is an easy-to-understand problem that has no visible symmetries but a hidden symmetry.
Therefore, it is well-suited to show that hidden symmetries can reduce problems that cannot be reduced by the common methods of model reduction using visible symmetries.
In addition, we present a general approach to reveal a hidden symmetry, which is the crucial part of this new concept.

This section is organized as follows. Subsection \ref{subsection - The Newsvendor Problem as an MDP} introduces the multi-period newsvendor problem and models it as an MDP.
Subsection \ref{subsection - applying the procedure} applies the procedure of reduction to it and briefly states the advantages of the reduction.
Finally, Subsection \ref{subsection - 5day newsvendor} applies the same procedure of reduction to the multi-period newsvendor problem with a 5-day cycle (i.e., the deviations depend on a 5-day cycle).

%Note that this problem actually does not require such a powerful concept. 

\subsection{Modeling the Multi-Period Newsvendor Problem} \label{subsection - The Newsvendor Problem as an MDP}

\paragraph*{Problem Description}
The multi-period newsvendor problem is a mathematical model that represents the situation of a newsvendor over an infinite number of days:
A newsvendor has to decide, based on a forecast, how many newspapers they are going to purchase for the next day. However, the real demand may differ from the forecast by a deviation and all unsold newspapers are worth nothing. 
These deviations behave independent and identically at each time step.
The newspapers have a fixed purchase and sale price. 		
The aim is now to determine the optimal purchase quantity for each day such that the newsvendor maximizes their expected profit.
So one can say, the multi-period newsvendor problem is infinitely often the newsvendor problem with different forecasts for each time step but with the same deviation distribution at each time step.

\paragraph*{Notation}
%For the multi-period newsvendor problem we use the following notation: 
%
	The \emph{time steps} or \emph{days} are denoted by $t\in \IN_0$.
	The \emph{deviation} at time step $t$ is denoted by $\delta_t\in \IZ$. 
	The $(\delta_t)_{t\in \IN_0}$ are i.i.d. random variables given by a probability distribution such that for all $t\in \IN_0$ the value $\delta_t$ is bounded. We denote the possible values of $\delta_t$ by $\IZ_\delta \subseteq \IZ$. 
	The \emph{forecast} at time step $t$ is a natural number $F_t$. 
	The forecasts are bounded such that there exists a constant $C$ with $0\leq F_t +\delta_t <C$ for all $t\in\IN_0$.
	The \emph{demand} at a time step $t$ is given by the sum 
	$D_t \doteq F_t + \delta_t.$
	The \emph{quantity of newspapers to be purchased at a time step $t$} is the number of \emph{ordered newspaper}, which is denoted by $O_t$. 
	We allow only $O_t \in F_t + \IZ_\delta$ because that is the range of the demand at time step $t$.
As reward, we simply use the daily profit:
We denote the \emph{purchase cost per newspaper} by $C_p \in \IR^+$ and the \emph{selling price per newspaper} by $C_s\in \IR^+$.
To form a reasonable model, the purchase cost must be less than the selling price, so we request $C_p < C_s$.		
The profit of the newsvendor at time step $t$ with order $O_t = a$ is given by the \emph{reward function} 
\begin{align*}						
	\reward_t(a)	\doteq  \min\myset{D_t, a} C_s - a C_p.
\end{align*} 
To be able to consider an arbitrary deviation in the reward function, we define the function
\begin{align*}						
	\reward_t(a,d)	\doteq  \min\myset{F_t + d, a)} C_s - a C_p,
\end{align*}
where $d\in \IZ_\delta$ describes the arbitrary deviation and $a$ the order of newspaper.
Obviously the relation $\reward_t(a) = \reward_t(a,\delta_t)$ holds.

The objective is to specify for all time steps $t$ an order quantity $O_t=a_t$ such that 
the expected profit $\expected{ \reward_t(a_t)}$ is maximized. 
%This is equivalent to the statement that the expected sum of discounted rewards
%\begin{align} \label{newsvendor problem objective}
%	\expected{ \sum_{t=0}^{\infty} \gamma^t \reward_t(a_t)}
%\end{align}
%is maximized. 

\begin{remark}
	Instead of using the daily profit as the reward function, any reward function that satisfies the following \emph{reward function condition} can be used: 
	
	Let $t$ be any time step and $\reward_t \colon (F_t+\IZ_\delta) \times \IZ_\delta \to \IR$ be the reward function at that time step of the multi-period newsvendor problem. 
	There exists an $x\in \IZ_\delta$ such that the statement
	\begin{align*}
		\forall t\in \IN_0:	x + F_t \in \argmax_{a \in F_t+\IZ_\delta} \reward_t(a,0)
	\end{align*}	
	is true
	and for all $t,l\in \IN_0$, $a\in F_t+\IZ_\delta$, $d\in\IZ_\delta$ the equation
	\begin{align*} 		
		\reward_t(a,d) - \reward_t(x + F_t,d)
		=
		\reward_{l}(a - F_t + F_{l},d) - \reward_{l}(x + F_{l},d)
	\end{align*}
	holds.
	
	For example, reward functions that punish deficits and surpluses more individually can also be used, as long as they fulfill the condition above. This is relevant, if for instance a certain level of safety in the sale is desired.	
\end{remark}

\paragraph*{The Problem as an MDP}
The MDP of the multi-period newsvendor problem $\mcMDP$ is defined as follows:
The set of states is given by $\mcS \doteq \mysetlong{s_t }{ t\in \IN_0}$, where $s_t$ represents the current time step $t$.		
The admissible actions in state $s_t$ are the possible orders of newspapers at time step $t$ and thus given by $\mcA(s_t) \doteq  F_t+ \IZ_\delta$.
This implies $\mcA \doteq \bigcup_{s\in\mcS} \mcA(s)$ and $\Psi \doteq \mysetlong{(s,a) }{ s\in \mcS, a\in \mcA(s)}.$				 
The set of rewards is the bounded set $\mcR \doteq \bigcup_{t\in \IN_0} \reward_t(F_t+\IZ_\delta, \IZ_\delta)$.		
For all  $r\in\mcR$, $t\in \IN_0$, $a \in\mcA(s_t)$, we define the set 
$$D(r,t,a) \doteq \mysetlong{d \in \IZ_\delta }{ r= \reward_t(a,d)},$$ 
which contains all possible deviation such that the reward $r$ is received in the state $s_t$ under action $a$.	
Using this set, the dynamic $p \colon \mcS \times \mcR \times \Psi \to [0,1]$ is defined as 
\begin{align*}
	p(s_l,r\vert s_t,a) \doteq 
	\begin{cases}
		\sum_{d\in D(r,t,a)} \Pr\myset{ \delta_t = d}, & \text{if }  l = t+1, \\
		0, &  \text{else.}
	\end{cases}
\end{align*}	
Hence, the state transition function $p \colon \mcS \times \Psi \to [0,1]$ is given by
\begin{align*}
	p(s_{l}\vert s_t,a) = 
	\begin{cases}
		1, & \text{if } l = t+1, \\
		0, & \text{else.}
	\end{cases}
\end{align*}

The objective in the multi-period newsvendor problem is to specify for all $t$ an order $a_t$ that maximizes $\expected{\reward_t(a_t)}$. This corresponds to finding an optimal policy in $\mcM$ with discount factor $\gamma=0$.
Hence, we set the discount factor to 0. Therefore, the value function is simply given by
\[ 	V_\pi^\mcM(s_t) = \expectedlong[\pi,p]{R_{1} }{ S_0=s_t} \]
for all $\pi\in\Pi(\mcM)$, $s_t\in\mcS$.

\begin{proposition}
The multi-period newsvendor problem $\mcM$ has no visible symmetries in general.
\end{proposition}
\begin{proof}
	Assume that $\mcM$ has a visible symmetry.
	Then there exist states $s_t$ and $s_l$ with $t\neq l$ which are symmetric to each other. This implies that $s_t$ and $s_l$ have the same reward function; thus $F_t= F_l$ holds. 
	Furthermore, the visible symmetry implies that the states $s_{t+k}$ and $s_{l+k}$ have the same reward function for all $k\in\IN$; thus $F_{t+k} = F_{l+k}$ holds for all $k\in\IN$.
	This is generally not true.
\end{proof}

\subsection{Applying the Procedure to Reduction}\label{subsection - applying the procedure}

The multi-period newsvendor problem was modeled as the MDP $\mcM$ and it has no visible symmetry. 
This subsection shows that the MDP $\mcM$ has a hidden symmetry and how to reduce the MDP with it.

\paragraph*{Outline of the Procedure}

The steps of the procedure of revealing and exploiting the hidden symmetry are as follows:
\begin{enumerate}
	\item 
	Solve the \emph{deterministic case} of the multi-period newsvendor problem, which is the problem without stochastics.
	The solution is a policy, which we denote by $\optdet$.
	
	\item Do a policy invariant reward shaping using the policy $\optdet$.
	
	\item Do a relabeling of the actions using the policy $\optdet$.
	
	\item Determine a hidden symmetry.
	
	\item Create the related quotient MDP. 
	
	\item Determine how an optimal policy in the quotient MDP is pulled back to the original MDP.
\end{enumerate}

\paragraph*{Step 1. Solution to the Deterministic Case}
In the \emph{deterministic case}, we assume $\delta_t =0$ for all $t\in\IN_0$.
Thus, the demand always equals the forecast 
(i.e. $F_t=D_t$),
and the reward function is given by
\begin{align*}
	\reward_t(a) = \min \myset{F_t,a}  C_s - aC_p.
\end{align*}
for all $t\in I$, $a\in\mcA(s_t)$.
Clearly, the reward is maximum for $a=F_t$.
Therefore, the optimal policy in the deterministic case is $\optdet\colon \mcS \to \mcA$ with $\optdet(s_t)= F_t$ for all $t\in \IN_0$.	
We will use this policy to align rewards and actions with it in a certain way.

\paragraph*{Step 2. Policy Invariant Reward Shaping}

We will change the reward structure of $\mcM\doteq (\mcS,\mcA,\Psi,\mcR,p)$ as follows:
If the decision maker is in the state $s_t\in \mcS$ and takes an action $a\in\mcA(s_t)$, the new reward received should be
\begin{align*}	
	\reward_t(a) - \reward_t(\optdet(s_t))
\end{align*} 
instead of $\reward_t(a)$. 
Since the deviation $\delta_t$ is hidden in the function $\reward_t \colon \mcA(s_t) \to \mcR$, the new reward in the state $s_t$ is modeled by 
\begin{align*}
	\mcA(s_t) \times \IZ_\delta & \longto \mysetlong{r_1 - r_2 }{ r_1,r_2 \in \mcR} \subseteq \IR \\
	(a,d) &\longmapsto \reward_t(a,d) - \reward_t(\optdet(s_t),d),
\end{align*}
where this new reward occurs with probability $\Pr\myset{\delta_t=d}$.

To define the MDP $\mcM'$ with this new reward structure, we need the following sets beforehand:
For all $t\in \IN_0$, $a\in  \mcA(s_t)$ and $r \in \mysetlong{r_1 - r_2}{ r_1,r_2 \in \mcR}$ we define the set
\begin{align*}
	D'(r, t, a) \doteq \mysetlong{ d\in \IZ_\delta }{ r = \reward_t(a,d) - \reward_t(\optdet(s_t),d) },
\end{align*}
which contains all possible deviations for which the new reward $r$ is received in the state $s_t$ under action $a$.

We define the MDP $\mcM'\doteq (\mcS,\mcA,\Psi, \mcR', p')$ with set of rewards
$\mcR' \doteq \mysetlong{r_1-r_2 }{ r_1,r_2\in \mcR} $
and the dynamic $p'\colon \mcS \times \mcR' \times \Psi\to [0,1]$ is given by
\begin{align*}
	p'(s_l,r\vert s_t,a) &\doteq 
	\begin{cases}
		\sum_{d\in D'(r,t,a)} \Pr\myset{\delta_t = d}, & \text{if }  l = t+1, \\
		0, &  \text{else.}
	\end{cases}
\end{align*}	
	
Hence, the state transition function $p'(s_l \vert s_t,a)$ is equal to 1 if  $l = t+1$, and 0 otherwise. Thus, $\mcM$ and $\mcM'$ have the same transition structure.

Since policies are defined over the set of state-action pairs, the two MDPs $\mcM$ and $\mcM'$ have the same set of policies; i.e. $\Pi(\mcM)=\Pi(\mcM')$.

\begin{lemma}
	The two MDPs $\mcM$ and $\mcM'$ have the same optimal policies. 
\end{lemma}
\begin{proof}
Because the discount factor is zero, an optimal policy $\pi$ in $\mcM$ maximizes 
$V_\pi^\mcM(s_t) \doteq \expected[\pi, p]{R_{1} \vert S_0=s_t}$  for all $s_t \in\mcS$.
Therefore, it also maximizes the term
\begin{align}\label{temp -- maximize term dings}
	\expected[\pi, p]{R_{1} \vert S_0=s_t} - \expected[\optdet, p]{R_{1} \vert S_0=s_t}
\end{align}
for all $s_t\in\mcS$.
By definition of $\mcM'$, the term (\ref{temp -- maximize term dings}) is equal to 
$\expected[\pi, p']{R_{1} \vert S_0=s_t} \doteq V_\pi^{\mcM'}(s_t)$.
Since $\mcM$ and $\mcM'$ have the same set of policies, a policy is optimal in $\mcM$ if and only if it is optimal in $\mcM'$; i.e. $\Pi_*(\mcM)  = \Pi_*(\mcM')$.
\end{proof}

In total, we did a policy invariant reward shaping because the MDPs $\mcM$ and $\mcM'$ have the same transition structure and the same optimal policies.

\begin{example}\label{example - 1 reward shaping step}
	Assume we have the forecasts $F_0=10$, $F_1=20$ and $F_2=15$ for days 0, 1, and 2, and the range of deviation given by $\IZ_\delta = \myset{-1,0,1}$. Further, let the purchase cost per newspaper be 5, and selling price per newspaper 7.
	
	The MDP is depicted in figure \ref{figure - example MDP 1}: 
	A node with the tuple $(t,F_t)$ as label represent the state $s_t$.
	At each normal state node, three outgoing arrows represent the three potential actions.
	The corresponding action is written on each arrow.
	Each arrow representing an action leads to a black node representing uncertainty. From the black node, there are three outgoing arrows representing the three potential deviations. The corresponding reward is written on each edge.		
	\begin{figure}[h]
		\centering
		\includegraphics[width=\textwidth]{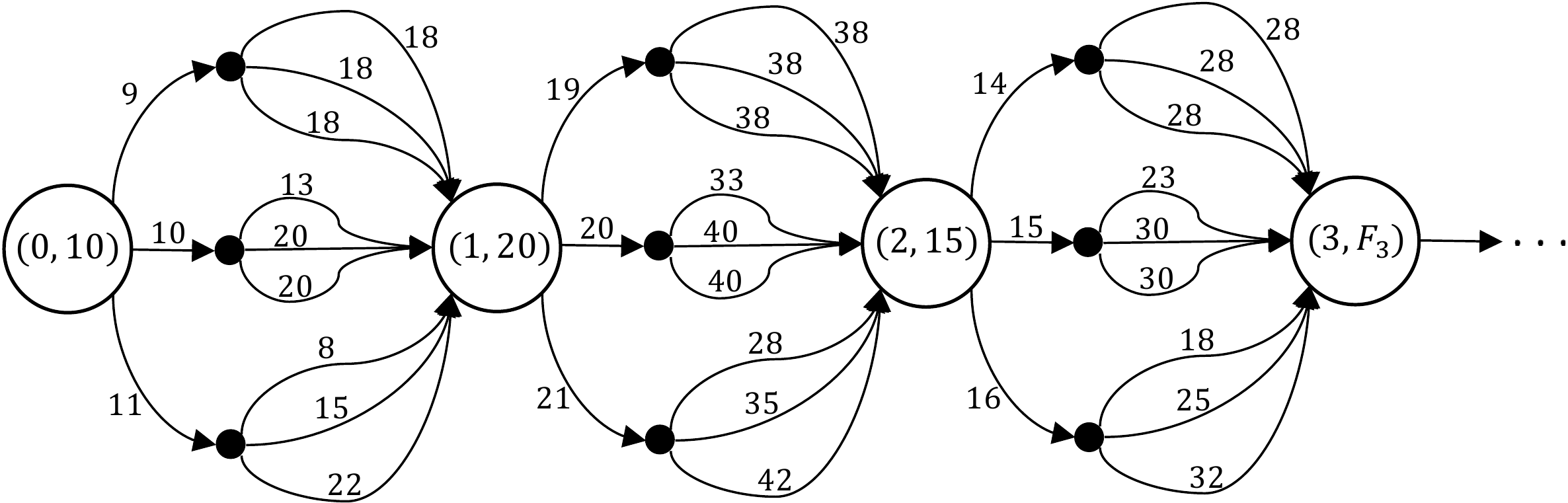}
		\caption{Example MDP without changed reward.}
		\label{figure - example MDP 1}
	\end{figure}
	
	After changing the reward structure as described in this section, we get a new MDP depicted in figure \ref{figure - example MDP 2}. This illustration already indicates the hidden symmetry. 
	It remains to adjust the actions.	
	\begin{figure}[h]
		\centering
		\includegraphics[width=\textwidth]{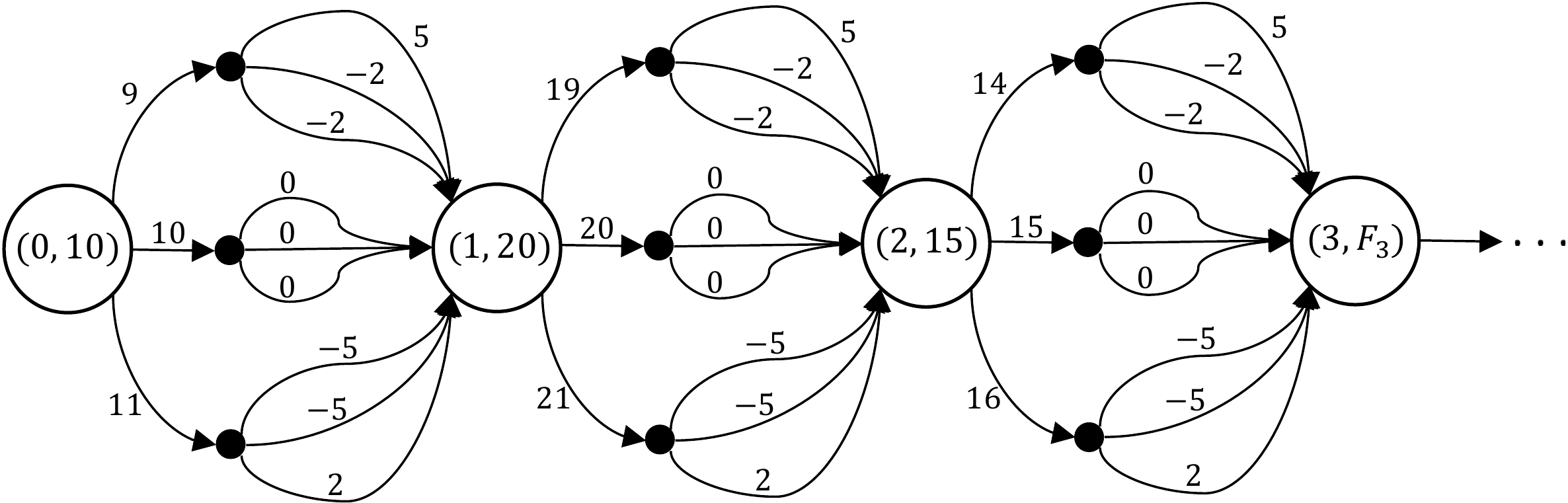}
		\caption{Example MDP with changed reward.}
		\label{figure - example MDP 2}
	\end{figure}
\end{example}

\paragraph*{Step 3. Relabeling of Actions}

We will relabel the actions of the current MDP $\mcM'\doteq (\mcS,\mcA,\Psi,\mcR',p')$ as follows:
In a state $s_t\in\mcS$, the new action $0$ shall be the optimal action of the deterministic case $\optdet(s_t)$, and the new action $n$ shall be optimal action of the deterministic case plus $n$, i.e. $\optdet(s_t) + n$. 
This means, we want to use a different denotation for the actions which is encoded by the bijective mapping
\begin{align*}
	g_{s_t} \colon \mcA(s_t) &\longto \mcA(s_t) - \optdet(s_t)\\ 
	a &\longmapsto a - \optdet(s_t)
\end{align*}
for all $s_t\in\mcS$.
Especially, $\mcA(s_t) - \optdet(s_t) = \IZ_\delta$ holds.

We get the relabeled MDP $\mcM''\doteq (\mcS,\mcA',\Psi',\mcR',p'')$ with $\mcA' \doteq \IZ_\delta$,  $\mcA'(s_t) = \IZ_\delta$, $\Psi' \doteq \mcS \times \IZ_\delta$ and the dynamic $p''\colon \mcS \times \mcR' \times \Psi' \to [0,1]$ is given by
\begin{align*}
	p''(s_l,r\vert s_t,a') \doteq p'(s_l,r\vert s_t,a' + \optdet(s_t)).
\end{align*}			

Since we only renamed the actions, the relabeling of $\mcM'$ to $\mcM''$ is given by the tuple $h\doteq(f, \mysetlong{g_{s_t} }{ s_t\in\mcS})$ with $f\colon \mcS \to \mcS$ given by $f(s_t)=s_t$. 		 
Therefore, $\mcM''$ is a relabeled variant of $\mcM'$, and by proposition 
\ref{prop - optimal policies relabelling},
the mapping
\begin{align*} 
	\Pi_*(\mcM'') \longto \Pi_*(\mcM'), \quad	\pi_* \longmapsto {\pi}_*\circ h
\end{align*}
%with ${\pi}_*'(a'\vert s_t) \doteq \pi_*(a' + \optdet(s_t)\vert s_t) $ for all $(s_t,a')\in \Psi'$
is bijective.

In total, we relabeled the actions. This was the last transformation needed to reveal a hidden symmetry of the newsvendor problem, as we will see.
Moreover, the MDPs $\mcM$ and $\mcM''$ are equivalent and $h$ is a relabeling from $\Psi$ to $\Psi'$ in respect to the equivalent definition of $\mcM$ and $\mcM''$.

\begin{example}[Example \ref{example - 1 reward shaping step} continued]\label{example - 2 relabelling of actions}
	After relabeling the actions as described in this subsection, we get a new MDP depicted in figure \ref{figure - example MDP after relabelling}. This figure clearly shows some symmetry in the MDP.	
	We see that the reward of an action is independent of the state.
	\begin{figure}[h]
		\centering
		\includegraphics[width=\textwidth]{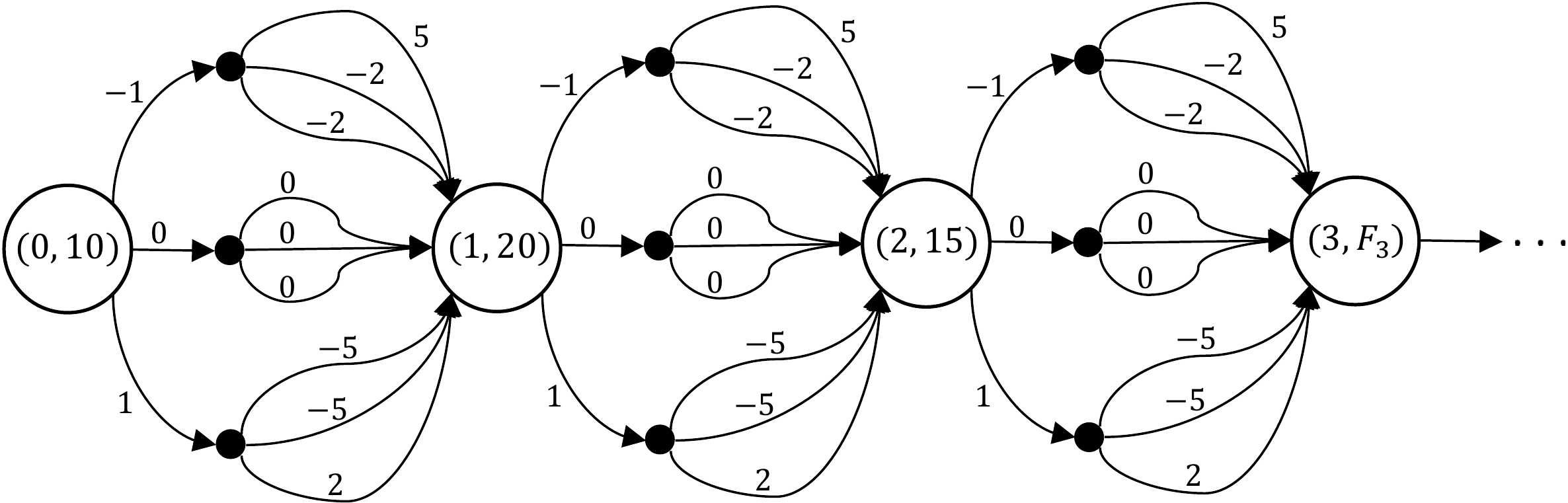}
		\caption{Example MDP after reward shaping and relabeling of actions.}
		\label{figure - example MDP after relabelling}
	\end{figure}
\end{example}

\paragraph*{Step 4. Determining a Hidden Symmetry}
As we can see in Example \ref{example - 2 relabelling of actions}, we revealed some kind of symmetry.
We capture this by the following lemma.
\begin{lemma}\label{lemma -- p'' ist immer gleich unabhaenging vom state}
	The equation
	\begin{align}\label{equation -- p'' ist immer gleich unabhaenging vom state}
		p'' (s_l, r \vert s_t, a') = p''(s_{l+k}, r \vert s_{t+k}, a')
	\end{align}
	holds for all 
	$t,l,k\in\IN_0$,
	$a\in\mcA'$,$r\in\mcR'$.
\end{lemma}
\begin{proof}
	If $l \neq t+1$, we get $0=0$. Let $l=t+1$.
	By definition of $p''$ and $p'$, we get
	\[ 
		p''(s_{t+1},r\vert s_t,a') 
		\doteq p'(s_{t+1},r\vert s_t,a'+F_t)
		\doteq \sum_{d\in D'(r,t,a'+F_t)} \Pr\myset{\delta_{t} =d}.
	\]
	If the set $D'(r,t,a'+F_t)$ is independent of $t$, then the equation (\ref{equation -- p'' ist immer gleich unabhaenging vom state}) is true because $(\delta_t)_{t\in I}$ are i.i.d. random variables.
	By definition, we have
	\begin{align*}
		D'(r, t, a'+F_t) \doteq \mysetlong{ d\in \IZ_\delta }{ r = \reward_t(a'+F_t,d) - \reward_t(F_t,d) }.
	\end{align*}
	We calculate
	\begin{align*}
		\reward_t(a' + F_t,d) - \reward_t(F_t,d)  =
		\left(\min\myset{ d, a' } - \min\myset{d, 0} \right) C_s - a' C_p,
	\end{align*}
	and thus, the set $D'(r,t, a'+F_t)$ is independent of $t$. 
	Hence, the equation (\ref{equation -- p'' ist immer gleich unabhaenging vom state}) is true.
\end{proof}

In Lemma \ref{lemma -- p'' ist immer gleich unabhaenging vom state}, we see that every state behaves the same.
Therefore, we define an equivalence relation $\sim_\mcS$ that identifies all states with each other, and an equivalence relation $\sim_{\Psi'}$ that identifies all state-action pairs with each other that have the same action.
Therefore
\begin{align*}
	\sim_{\mcS} &\doteq \mcS \times \mcS\\
	\sim_{\Psi'} &\doteq \mysetlong{ ((s_1,a_1),(s_2,a_2)) \in \Psi' \times \Psi' }{ a_1= a_2 }.
\end{align*}
The tuple $\sim \doteq (\sim_{\mcS},\sim_{\Psi'})$ is a candidate for a visible symmetry.
In fact, $\sim$ is a simple visible symmetry in $\mcM''$ because for all $(s_t,a_t) \sim_{\Psi'} (s_l,a_l)$ the statement
\[ 
\forall r\in\mcR' : p''(\mcS, r \vert s_t,a_t) = p ''(\mcS, r \vert s_l,a_l)
 \]
is true (because of Lemma \ref{lemma -- p'' ist immer gleich unabhaenging vom state}), and $a_t=a_l$ holds.

Since the MDPs $\mcM$ and $\mcM''$ are equivalent, the tuple $(\mcM'',\sim)$ is a hidden symmetry of $\mcM$.

\paragraph*{Step 5. Quotient MDP}
With the hidden symmetry $(\mcM'', \sim)$ of the MDP $\mcM$, we can create the associated quotient MDP $\quotsimMprimeprime$.
The quotient MDP $\quotsimMprimeprime \doteq ( \olS, \olA, \olPsi, \olR, \olp )$ is given as follows:
The set of states is $\olS \doteq \myset{\mcS}$.
The set of actions is $\olA \doteq \IZ_\delta$, and the set of admissible state-actions pairs is $\olPsi \doteq \myset{\mcS} \times \IZ_\delta$.	The set of rewards is $\olR \doteq \mcR'$.
 The dynamic $\olp\colon \olS \times \olR \times \olPsi \to [0,1]$ is given by 
\begin{align*}
	\olp(\mcS,r \vert \mcS,a) \doteq p''(s_1,r \vert s_0,a).
\end{align*}	

Hence, we created the quotient MDP $\quotsimMprimeprime$, which is very simple since it consists of only one single state. 
In a sense, the quotient MDP reflects the stochastic of the original problem.
	
\begin{example}[Example \ref{example - 2 relabelling of actions} continued]
	By using the simple visible symmetry, we can aggregate all states into one and form the corresponding quotient MDP.
	The quotient MDP consist of only one state as depicted in figure \ref{figure - example MDP 4}.
	\begin{figure}[h]
		\centering
		\includegraphics[width=0.53\textwidth]{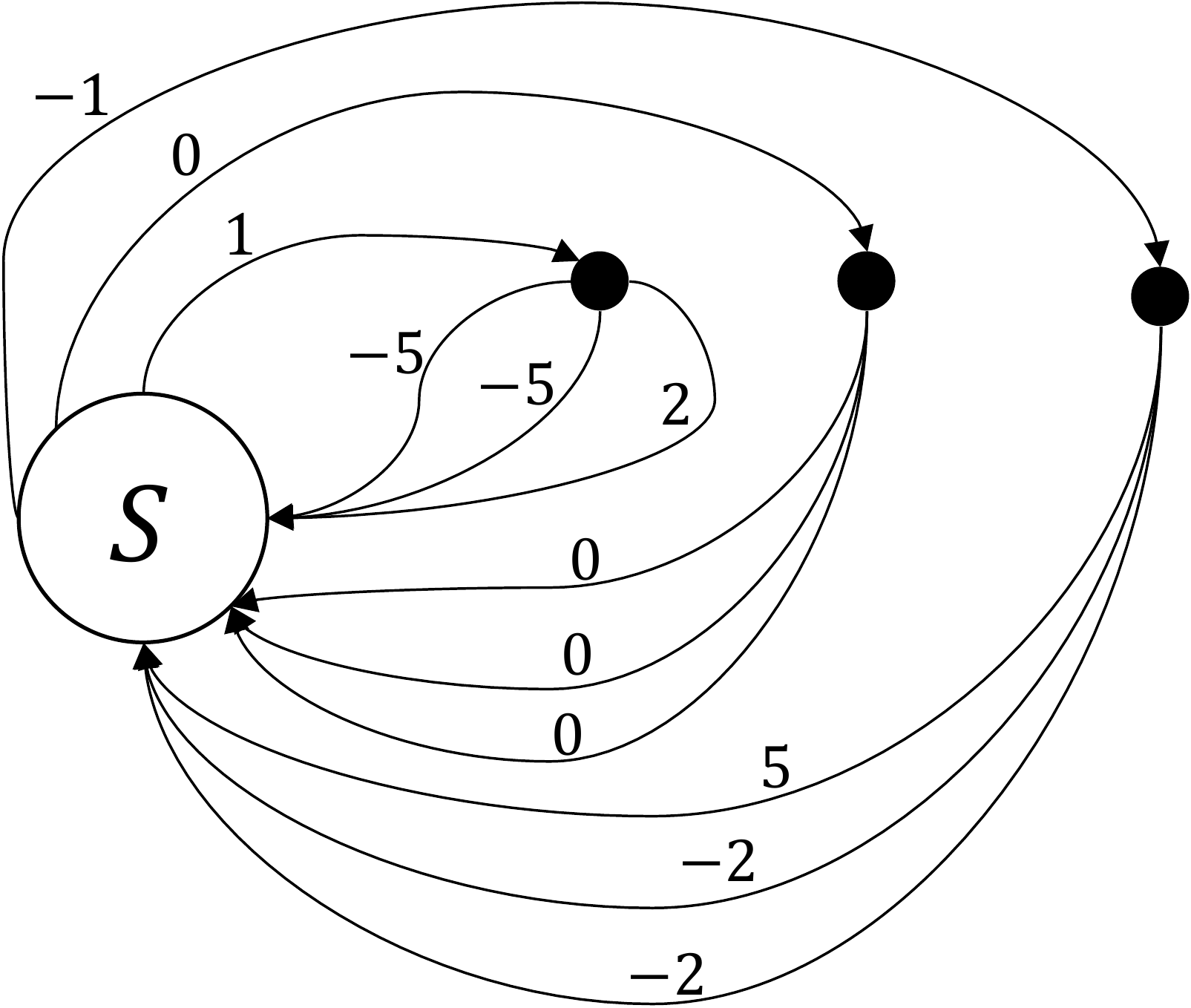}
		\caption{Quotient MDP of the example MDP.}
		\label{figure - example MDP 4}
	\end{figure}	
\end{example}

\paragraph*{Step 6. Pullback of Optimal Policies}

To complete the reduction of the multi-period newsvendor problem, we state how to pull an optimal policy in the quotient MDP $\quotsimMprimeprime$ back to the original MDP $\mcM$.

By Theorem \ref{thm - hidden symmetry pullback of optimal policy}, we get the injective mapping
$\Lambda \colon \Pi_*(\quotsimMprimeprime) \to \Pi_*(\mcM)$	
given by
\begin{align*}		
	\Lambda(\pi)(a\vert s_t) \doteq 
	\pi(a - \optdet(s_t)\vert [s_t])
	= \pi(a - F_t \vert \mcS)
\end{align*} 
for all $(s_t,a)\in\Psi$.
Moreover, the sets $\Pi_*(\quotsimMprimeprime)$ and $\Pi_*(\mcM)$ are not empty (this holds by \cite[Theorem 6.2.10]{Puterman2014}).

By considering the map $\Lambda$ with a deterministic policy $\pi \in \Pi_*(\quotsimMprimeprime)$ as an argument, one gets a better intuition of the pullback of the optimal policies from $\quotsimMprimeprime$ to $\mcM$. Then, $\Lambda(\pi)$ is a deterministic optimal policy in $\mcM$ and is given by
\[ 	\Lambda(\pi) (s_t) \doteq	\pi(\mcS) + \optdet(s_t) = \pi(\mcS) + F_t. \]

This pullback completes the reduction of the multi-period newsvendor problem using a hidden symmetry.

\paragraph*{Advantages of the Reduction}
%\subsection{Theoretical Advantages}\label{subsection - theoretical advantages}
The multi-period newsvendor problem $\mcM$ has been significantly reduced:
We reduced the MDP $\mcM$, which consists of infinitely many states, to the MDP $\quotsimMprimeprime$, which consists of only one state.
Moreover, the reduced MDP $\quotsimMprimeprime$ is as complex as one state of $\mcM$.

Both MDPs are \emph{bandit problems} \citep{Sutton2017}, and thus, it is easy to
compare them: The MDP $\mcM$ is a contextual bandit problem with $\abs{\mysetlong{F_t}{t\in\IN_0}}$ multi-
armed bandits, and the MDP $\quotsimMprimeprime$ is just one multi-armed bandit. 
Therefore, the MDP $\mcM$ needs at least $\abs{\mysetlong{F_t}{t\in\IN_0}}$ times as many steps as the MDP $\mcM$ to be solved with the same bandit solving algorithm, depending on the occurrence of the values in the sequence $(F_t)_{t\in\IN_0}$.

\subsection{Problem Variant: The 5-day Cycle Multi-Period News\-vendor Problem}\label{subsection - 5day newsvendor}

This subsection applies the same procedure of reduction as above to a variant of the multi-period newsvendor problem.
This variant has deviations that follow a cyclic behavior. 
In this way, we can better understand the role of deviations for our procedure of reduction and hidden symmetries.
We will see that this cyclical property of the deviations is sufficient.

\paragraph*{Applying the Reduction Procedure}
The \emph{5-day cycle multi-period newsvendor problem} follows a 5-day cycle for the deviations as follows. 
We assume that the newsvendor works 5 days a week, say Monday to Friday, and the deviations depend on the respective days of the week.
Thus, the deviations of each set  $D_b \doteq \mysetlong{\delta_{b+5k} }{ k \in \IN_0 }$ with $b\in \myset{0,1,2,3,4}$ are i.i.d.\ random variables.

As in the standard multi-period newsvendor problem, the optimal policy of the deterministic case is again $\optdet$, and the transformations are done in the same way. 
Therefore, we obtain the MDP $\mcM'' \doteq(\mcS, \mcA',\Psi',\mcR',p'')$ as described in Subsection \ref{subsection - applying the procedure}. 
The only difference is that 
the equation 
$p'' (s_l, r \vert s_t, a') = p''(s_{l+k}, r \vert s_{t+k}, a')$
from Lemma \ref{lemma -- p'' ist immer gleich unabhaenging vom state}
does not hold for all $k\in\IN_0$. 
Instead the equation only holds for all $k\in 5\IN_0$.
This is because the deviations of each set $D_b$ 
%$D_b\doteq \mysetlong{\delta_{b+5k} }{ k \in \IN_0}$ 
with $b\in \myset{0,1,2,3,4}$ are i.i.d.\ random variables.

In contrast to the original multi-period newsvendor problem, 
the simple visible symmetry of $\mcM''$ here does not aggregate all states/days into one but all Mondays, Tuesdays, \dots, Fridays. 
The aggregations of the state-action pairs are done in the same way as in the original multi-period newsvendor problem but with the addition that the states must have to be the same day of the week.
Mathematically, the simple visible symmetry is given by $\sim\doteq(\sim_{\mcS}, \sim_{\Psi'})$ with  
\begin{align*}
	\sim_{\mcS} \doteq& \mysetlong{(s_{t_1}, s_{t_2}) \in\mcS\times\mcS }{ t_1-t_2 \in 5\IZ} \\
	\sim_{\Psi'} \doteq& \mysetlong{((s_{t_1},a_1'),( s_{t_2},a_2')) \in\Psi'\times\Psi' }{ t_1-t_2 \in 5\IZ, a_1' =a_2'}.
\end{align*}
Due to this, the quotient MDP consists of 5 states.
The pullback of optimal policies from the quotient MDP to the original MDP is done via the same map
$\Lambda$.

\paragraph*{Gained Insights}
In the normal variant, the deviations are i.i.d. random variables. 
It turns out that this condition is not required for our procedure of reduction and the existence of a hidden symmetry.
Instead, a cyclic behavior of the deviations, as in this variant, is sufficient. 

The cyclic behavior causes the quotient MDP to have five states instead of one.
Interestingly, the pullback is completely the same.

\section{Conclusion}\label{section - conclusion}

This paper introduced hidden symmetries, which are a new type of symmetries for MDPs.
They naturally provide a model reduction framework for MDPs. 
We showed how this model reduction framework using hidden symmetries works by applying it to the multi-period newsvendor problem.
In this way, we showed that hidden symmetries can reduce problems that visible symmetries cannot reduce. This is because the multi-period newsvendor problem has no visible symmetries but a hidden symmetry.
We drastically reduced the problem by using hidden symmetries. The standard variant has a reduced MDP consisting of only one state.
This highlights the advantage of hidden symmetries over the common visible ones.

Our procedure of reduction of the multi-period newsvendor problem followed some concrete steps.
The main idea here was to equivalently transform the MDP based on a fixed policy.
The presented procedure took advantage of the fact that the discount factor is zero.
This caused that the policy invariant reward shaping step was simple.

We suppose that the approach of using a fixed policy to reveal a hidden symmetry will work for many other problems as well. 
For example, we think that planning problems modeled as control problems in discrete-time dynamical systems will benefit from this approach.
Clearly, if the discount factor is greater than zero, the policy invariant reward shaping step must be adjusted.
Currently, we are working on such a generalized version of the procedure presented in this paper.

%\section*{Declarations}
%The authors declare that they have no conflict of interest.


\bibliography{library}% common bib file


\end{document}